\documentclass[letterpaper,11pt,reqno]{amsart} 
\usepackage[portrait,margin=3cm]{geometry} 

\usepackage{mathrsfs,xfrac} 
\usepackage[colorlinks=true,linkcolor=blue,citecolor=blue,urlcolor=blue]{hyperref} 
\usepackage{amsmath,amssymb,amsthm,amsfonts,amsbsy,latexsym,dsfont,color,graphicx,enumitem}
\usepackage[numeric,initials,nobysame]{amsrefs} 
\usepackage[foot]{amsaddr}
\usepackage{caption}
\usepackage[textsize=tiny]{todonotes}
\usepackage{regexpatch}
\usepackage{float}
\makeatletter
\xpatchcmd{\@todo}{\setkeys{todonotes}{#1}}{\setkeys{todonotes}{inline,#1}}{}{}
\usepackage{comment}
\makeatother

\newtheorem{thm}{Theorem}[section]
\newtheorem{lem}[thm]{Lemma}

\newtheorem{prop}[thm]{Proposition}

\newtheorem{conj}[thm]{Conjecture}

\newtheorem{remark}{Remark}
\newtheorem{ass}{Assumption}

\newtheorem{ques}[thm]{Question}

\renewcommand{\le}{\leqslant}  
\renewcommand{\ge}{\geqslant}

\newcommand{\half}{\ensuremath{\sfrac12}} 

\newcommand{\eps}{\varepsilon}

\newcommand{\abs}[1]{\left\vert#1\right\vert}

\newcommand{\ie}{\emph{i.e.,}}

\newcommand{\eg}{\emph{e.g.,}}

\let\ga=\alpha   \let\gd=\delta 
     \let\gl=\lambda     \let\gn=\nu    \let\gs=\sigma \let\gt=\tau

\newcommand{\cI}{\mathcal{I}}

\newcommand{\cR}{\mathcal{R}}

\newcommand{\cV}{\mathcal{V}}
  
\newcommand{\vone}{\mathbf{1}}

\newcommand{\mvgn}{\boldsymbol{\nu}}

\newcommand{\dN}{\mathds{N}}

\newcommand{\dZ}{\mathds{Z}} 
\DeclareMathOperator{\E}{\mathds{E}}
\DeclareMathOperator{\pr}{\mathds{P}}

\newcommand{\wh}[1]{\widehat{#1}}
\newcommand{\ol}[1]{\overline{#1}}
\newcommand{\mix}{\mathrm{mix}}
\newcommand{\dist}{\mathrm{dist}}

\begin{document}
\title[Collaboration of Random Walks on Graphs]{Collaboration of Random Walks on Graphs}
\author[Dey]{Partha S.~Dey$^1$}
\author[Kim]{Daesung Kim$^2$}
\author[Terlov]{Grigory Terlov$^1$}

\address{$^1$University of Illinois Urbana--Champaign, 1409 W Green Street, Urbana, Illinois 61801}
\address{$^2$Georgia Institute of Technology, 686 Cherry Street, Atlanta, Georgia 30332}
\email{psdey@illinois.edu, dkim3009@gatech.edu, gterlov2@illinois.edu}

%\date{\today}

\begin{abstract}
Consider a collaborative dynamic of $k$ independent random walks on a finite connected graph $G$. We are interested in the size of the set of vertices visited by at least one walker and study how the number of walkers relates to the efficiency of covering the graph. To this end, we show that the expected size of the union of ranges of $k$ independent random walks with lifespans $t_1,t_2,\ldots,t_k$, respectively, is greater than or equal to that of a single random walk with the lifespan equal to $t_1+t_2+\cdots+t_k$. We analyze other related graph exploration schemes and end with many open questions.
\end{abstract}
\maketitle
%%%%%%%%%%%%%%%%%%%%%%%%%%%%%%%%%%%%%%%%%%%%%%%%%%%%%%%%%%%%%%%%%%%%%%%
\section{Introduction}
%%%%%%%%%%%%%%%%%%%%%%%%%%%%%%%%%%%%%%%%%%%%%%%%%%%%%%%%%%%%%%%%%%%%%%%
Across many disciplines, one often has to work with a complex network given only local information. While exact objectives may vary, they usually reduce to the problem of discovering the graph efficiently. Consider the following general setup. An agent is placed on a site in an unknown network, and at each given moment, they may travel to a neighboring site of their current position. Naturally, one is interested in the evolution of the set of visited vertices and how it approximates the properties of the underlying network. One can visualize this problem by imagining a person placed in an unknown city and tasked with drawing its map. What is the best strategy to proceed with? Would it help to place several people to do this job together? If one knows that this city looks like a square grid of size $m\times n$, then one can easily traverse it in $mn$ steps. However, it is not easy to devise a deterministic strategy for such a cartographer without knowing the specific properties of the network. Hence, naturally, one would like to implement a randomized strategy. In economics, it may be used to optimize marketing strategies on a social network, also known as seeding (see \eg~\cites{Akbarpour20,sadler22}). In mathematical biology, where the graph is based on the interaction of the species, random walks are used to explore the communal and hierarchical structures (see \eg~\cites{Rosvall08, Rosvall11}).

Similarly, in the study of gerrymandering, one is interested in understanding the set of all possible partitions (maps) of a given area into voting districts that satisfy certain conditions. Identifying the set of outliers is a difficult problem. While this set is highly complex, letting two maps be connected in a network if they differ only in a small location equips this set with a natural metric. This enables the exploration of the network with a random walk, leads to practical sampling, and gives further insight into various structural properties (see \eg~\cite{DeFord21} and references therein).

There have been considerable efforts to understand random sets arising from random walks on graphs. For instance, the fluctuation behavior of the set of vertices visited by random walks was studied in~\cites{DE51, JP70clt, JO68, LeGall} for $\dZ^d$, and~\cite{DK21} for the discrete torus $\dZ^d_n$. Sznitman~\cite{Sznitman10} introduced the random interlacement model,  constructed using a Poisson point process on the set of doubly infinite nearest--neighbor paths on $\dZ^d$. The percolative properties of the range of a simple random walk on $\dZ^d_n$ have been investigated in~\cites{Sznitman10, TW11}, combining the random interlacement model with coupling techniques. Another example is the competition of random walks on a graph (or the coloring of a graph with random walks), studied in~\cites{Gomes96, Dicker06, Miller13}. In a competitive dynamic, each walker is associated with the set of vertices it visited before others. Hence, several randomly growing sets compete for the area in this model.

This paper proposes a collaborative dynamic among several independent random walks on a graph. In contrast with the competition of random walks described above, a vertex is considered discovered if at least one of the random walks visited it. We investigate the average size of this set compared to that of a single random walk and its relation to the distribution of starting positions. In particular, Theorem~\ref{thm: k-collab} says that on average, $k$ independent identical walkers with lifespans $t_1,t_2,\ldots,t_k$ started from the stationary distribution independently of each other, on average discover a higher proportion of the graph than a single walker, also started at the stationary distribution, with the lifespan $t_1+t_2+\cdots+t_k$. This holds for any finite connected graph $G$ and any $k\in \dN$, very mild assumptions on $t_i$ (if any), and any time-homogeneous random walk with a stationary distribution on $G$. In Lemma~\ref{thm: k-collab vs star} we show that the former collaboration scheme also, on average, discovers a higher proportion of vertices than the same amount of independent random walks started at independently chosen vertices following the same distribution (``star" shape).

Our results lead to various natural questions, including a quantitative version of the inequality from Theorem~\ref{thm: k-collab}, stochastic dominance of one collaboration scheme over the other, and how the single walker scheme compares to the ``star" shaped one. We explicitly state and discuss these questions in Section~\ref{sec: discussion}. First, we review some basic results about reversible Markov chains.

\subsection{Markov chains on graphs and networks}
It is well-known that a reversible Markov chain on a finite state space can be seen as a random walk on an edge-weighted network. Let $G=(V,E)$ a finite undirected connected graph $G$. For each edge $e\in E$, we assign the weight $c(e)>0$, also known as \textit{conductance} of $e$. For $x,y\in V$ we write $x\sim y$ if $(x,y)\in E$. In this paper, we consider the Markov chain $(X(t))_{t\ge 0}$ on $G$ with transition matrix $P=(P(x,y))_{x,y\in V}$ given by
\begin{align*}
  P(x,y) = \begin{cases}
    \frac{c(x,y)}{c(x)}, & (x,y)\in E,\\
    0, & \text{otherwise},
  \end{cases}
\end{align*}
where $c(x)=\sum_{y:y\sim x}c(x,y)$.  One can interpret these weights as the inscription of the bias of the walker. For example, if one would like to prioritize discovering vertices with higher degrees, one can set $c(x,y)$ to equal the indicator that $(x,y)$ is an edge. For further information on random walks on networks, we refer the interested reader to~\cites{LPW, LP}. 

Note that $X(t)$ has the unique stationary distribution $\pi$ given by
\begin{align*}
  \pi(x) = \frac{c(x)}{\sum_{y\in V}c(y)} \quad\text{ for }x\in V.
\end{align*}

\subsection{Set up}\label{sec:setup}
Throughout the paper $k$ is a fixed positive integer and $(X_i(t))_{t\in \cI}$, for $i\in[k]:=\{1,2,\ldots,k\}$ and an appropriate index set $\cI$, are \textit{independent} Markov chains that satisfy the following assumption.
\begin{ass}\label{ass:mc}
For all $i\in [k]$ Markov chains  $(X_i(t))_{t\in \cI}$ have the \textit{same} transition matrix on a finite state space $V$ of size $\abs{V}=n$. Each $X_i$ is assigned a corresponding length of trajectory $t_i$, which we call a lifespan. We further assume that the Markov chain is \textit{time-homogeneous, reversible, irreducible, and aperiodic}. 
\end{ass}

Equivalently, each $X_i(t)$ can be seen as a random walk on the corresponding network $G=(V,E)$.  Let $\pi$ be the stationary distribution for the Markov chain. The range of the $i$-th random walk $X_i(t)$ at time $t$ is denoted by
\begin{align*}
	\cR_{i}(t):=\{X_{i}(s)\mid s\le t\} \quad \text{ for } i\in[k].
\end{align*}
If $k=1$, we simply denote by $X(t)=X_1(t)$ and $\cR(t)=\cR_1(t)$. 

Let $\nu_i$ be probability measures on $V$, for $i\in[k]$. We use the notations $\pr_{\nu_1,\nu_2,\ldots, \nu_k}$ and $\E_{\nu_1,\ldots, \nu_k}$ to denote the probability and the expectation of $(X_i)_{i\in[k]}$ where $(X_1,X_2,\ldots,X_k)$ starts from $\nu_1\otimes\cdots\otimes\nu_k$. If $\nu_i$'s are the same distribution $\nu$, we simply write $\pr_{\nu^k}$ and $\E_{\nu^k}$. If we sample the same starting points for all the $k$ random walks $(X_i)_{i\in[k]}$ from a single probability measure $\nu$, the probability, and the expectation are denoted by $\pr_\nu$ and $\E_\nu$. We use the notation $\pr_{x\sim \nu}$ and $\E_{x\sim\nu}$ to specify the starting point $x$ chosen from the distribution $\nu$. When a Markov Chains starts at a given point $x\in V$, {\ie} $\nu=\gd_x$, we 
denote the corresponding expectation and probability as $\E_x$ and $\pr_x$, respectively.

For some of our results, we further assume that $(X_i(t))_{t\in \cI}$ falls into one of the following three cases.
\begin{enumerate}[label=(\roman*)]
    \item for each $i\in [k]$, $(X_i(t))_{t\ge 0}$ are continuous-time Markov chains driven by an exponential clock with intensity $1$, \label{continuous version}
    \item for each $i\in [k]$, $(X_i(t))_{t\in \dN}$ are $\half$--lazy Markov chains, {\ie} at each step with probability $\half$ the particle does not move and otherwise proceeds following its transition matrix,\label{lazy version}
    \item for each $i\in [k]$, $(X_i(t))_{t\in \dN}$ are discrete time Markov chains and $t_1+t_2+\cdots+t_k$ is even. \label{even version}
\end{enumerate}
We believe that the last case is a restriction caused by our methods and should not play a significant role in the behavior of the system, see Conjecture~\ref{conj: oddcase}.

\subsection{Main results}\label{sec:main}

\begin{thm}[One vs.~Many - Uniform I.I.D.]\label{thm: k-collab}
    Let $k\ge 2$, $t_1,t_2,\ldots,t_k\in \cI$, and $(X_i(t))_{t\in\cI}$ be Markov chains on a finite connected graph $G$ satisfy Assumption~\ref{ass:mc} and falls into one of the cases~\ref{continuous version},~\ref{lazy version}, or~\ref{even version}. Then we have
\begin{align}\label{eq: main}
  \E_{\pi^k} \abs{\bigcup_{i=1}^k\cR_i(t_i)} \ge\E_{\pi} \abs{\cR\left(\sum_{i=1}^k t_i\right)} .
\end{align}
\end{thm}
The result shows that the expected size of the set of vertices covered by multiple random walks is always as large as that of a single random walk. Note that the inequality is not an asymptotic estimate and holds for any timespans $t_1,t_2,\ldots,t_k>0$. The proof is based on the spectral representation of the survival probabilities.

For some graphs, it might be challenging to compute the stationary measure of a given random walk; hence, one would want to relax the assumption on the distribution of the starting positions. 

For example, it is well-known that the distribution of the position after a time of order of the mixing time with an additional log factor is close to the stationary. It is natural to ask if one acquires starting positions in such a fashion, the statement of Theorem~\ref{thm: k-collab} would still hold. Propositions~\ref{lem: nonstat} and~\ref{lem:almostind} give affirmative answers. Define the following quantity 
\begin{equation}\label{eq:pistar}
    \pi_\ast := \min_{x\in V}\frac{\pi(x)}{1-\pi(x)}=\frac{\pi_{\min}}{1-\pi_{\min}}>0,
\end{equation}
where $\pi_{\min}:=\min_{x\in V}\pi(x)$.
Notice that, if $\pi$ is uniform, then $\pi_\ast = (|V|-1)^{-1}$ where $|V|$ is the number of vertices. Suppose $\pi$ is not uniform, then there exists a vertex $y$ such that $\pi(y)<1/|V|$. Since the map $t\mapsto t/(1-t)$ is increasing, we see that $\pi_\ast<(|V|-1)^{-1}$. Thus, we conclude that $\pi_\ast$ is maximized when $\pi$ is uniform, and the deficit of $\pi_\ast$ from its maximum value $(|V|-1)^{-1}$ measures how much $\pi$ is away from uniformity.

\begin{prop}[One vs.~Many - Near Uniform Independent]\label{lem: nonstat}
Under the same assumptions as in Theorem~\ref{thm: k-collab}. Suppose $\pi_\ast$ is as in~\eqref{eq:pistar} and $\nu_i$ be probability measures on $V$ for $i\in[k]$ such that 
\begin{align}\label{eq:nu-assumption1/3}
  \sup_{x\in V}\left|\frac{\nu_i(x)-\pi(x)}{\pi(x)}\right|\le (1+ \pi_\ast)^{1/k}-1\approx \pi_\ast/k. 
\end{align}
Then for every $t_j\in\cI$, $j=1,2,\ldots,k$, we have
\begin{align*}
  \E_{\nu_1, \nu_2, \ldots, \nu_k} \abs{\bigcup_{i=1}^k\cR_i(t_i)} 
  \ge\E_{\pi}\abs{\cR\left(t_1+t_2+\cdots+t_k\right)} .
\end{align*}
\end{prop}

\begin{remark}
Consider the $d$-dimensional discrete torus $\dZ^d_n$ of side-length $n$ for $d\ge 5$. It is well-known that if $t= \ga n^2\log n$, then 
%Then the distribution $\mu(x)=\pr_0(X(t)=x)$ of the location at time $t$ satisfies
%\begin{align*}
%    \|\mu - \pi\|_{TV}\le c_1 n^{-c_2}
%\end{align*}
%for some $c_1,c_2>0$, 
\begin{align*}
    \max_{x,y\in \dZ^d}\left|\frac{\pr_x(X(t)=y)}{\pi(y)}-1\right|\le c e^{-c\ga}
\end{align*}
for some $c>0$, where $\pi$ is the stationary measure. Recall that for $\dZ^d_n$ the size of the vertex set $\abs{V}=n^d$. Since $|\cR(t)|/|V|\le t/n^d=\ga n^{2-d}\log n$, we see that if $n$ and $\ga$ are large enough, then $\pr_x(X(t)=y)$ is close to the stationary measure while the random walk rarely covers $G$ up to time $t$. 
\end{remark}

\begin{prop}[One vs.~Many - Near Uniform Dependent]\label{lem:almostind}
Under the same assumptions as in Theorem~\ref{thm: k-collab}. Suppose $\pi_\ast$ is as in~\eqref{eq:pistar}  and $\nu_i,i\in[k]$ are probability measures on $ V$ such that for some $\eps\in (0,1)$
\begin{align}\label{eq:nu-assumption1/4}
  \sup_{x\in V}\left|\frac{\nu_i(x)-\pi(x)}{\pi(x)}\right|\le (1+\eps \pi_\ast)^{1/k}-1, \text{ for all } i\in[k].
\end{align}
Assume that, there is a coupling $\mu$ of $\nu_i,i\in[k]$ such that
\begin{align}\label{eq:almost ind}
  \sup_{x_i \in G, i\in[k]}\frac{|\mu(x_1,x_2,\ldots,x_k)-\prod_{i=1}^k\nu_i(x_i)|}{\prod_{i=1}^k\pi(x_i)}\le (1-\eps)\cdot \pi_\ast.
\end{align}
Then for every $t_1, t_2,\ldots,t_k\in\cI$, we have 
\begin{align*}
  \E_{\mu} \abs{\bigcup_{i=1}^k \cR_i(t_i)} 
  \ge\E_{\pi} \abs{\cR\left(t_1+t_2+\cdots+t_k\right)}.
\end{align*}
\end{prop}

We remark that the Assumptions~\eqref{eq:nu-assumption1/3},~\eqref{eq:nu-assumption1/4}, and~\eqref{eq:almost ind} are quite strong in the sense that $\pi_\ast\le (|V|-1)^{-1}$ and we are interested in the case where the size of the graph $|V|$ is large. However, such assumptions are required since the results in Propositions~\ref{lem: nonstat} and~\ref{lem:almostind} hold for \emph{any} number of random walks, \emph{any} underlying network, and very mild assumptions on $t_i$, if any (depending on the cases~\ref{continuous version},~\ref{lazy version}, and~\ref{even version}). This observation also suggests that not only stationary but also the uniformity of the starting distribution is crucial in this context. Hence one can also ask if the converse is true in some sense: if multiple random walks are more efficient in covering the graph than a single random walk for any number of random walks and for any lengths of time spans, then the starting distribution should be almost stationary and almost uniform in a practical sense. 

More generally this can be stated as the following question: \emph{how does the average size of the total area covered by $k$ walkers depend on the joint distribution of the initial positions?} See Question~\ref{q:opt}. While the generality of this question makes it challenging, it turned out that we can answer the question for particular schemes. In particular, we show that $k$ random walks started at points, chosen independently from the same distribution $\nu$, on average cover the graph more effectively than the $k$ walkers started at the same point, sampled from $\nu$.

\begin{lem}[Star vs. Many IID]\label{thm: k-collab vs star}
For any $k,t_1=t_2=\cdots=t_k\in \dN$, connected graph $G$, $(X_i(t))_{t\in\dN}$ a random walk on it, and a probability measure $\gn$ on $V$ we have 
\begin{align*}
  \E_{\gn^k} \abs{\bigcup_{i=1}^k\cR_i(t_i)} \ge\E_{\gn} \abs{\bigcup_{i=1}^k\cR_i(t_i)}.
\end{align*}
\end{lem}

We have shown that, under appropriate assumptions, multiple random walks starting with independent (almost) stationary distributions are better than a single random walk or starting at one vertex provided. A natural further step is to compare a single random walk and multiple random walks with the single starting position; see Question~\ref{q:star vs path}. The following theorem answers this question in the case of $d$-dimensional discrete torus $\dZ^d_n$, $k=3$, and such that at least one lifespan is long enough.

\begin{thm}\label{thm:starone}
Let $G=\dZ^d_n$, $d\ge 3$, and $t_1, t_2, t_3\in\dN$. 
Consider three independent simple random walks $X_1, X_2, X_3$ starting at $0$.

Then there exists a constant $C>0$ such that if $t_1\ge C n^2\log n$, then 
\begin{align*}
  \E_{0} \abs{\cR_1(t_1)\cup \cR_2(t_2)\cup \cR_3(t_3)}\ge \E_{0} \abs{\cR\left(t_1+t_2+t_3\right)}.
\end{align*}
\end{thm}

The main ingredients of the proof are time reversal and the fact that the mixing time is of order $n^2$. One may extend the result to more general settings where time reversal and mixing properties are available. Moreover, the case when $k\ge 4$ remains open.

%%%%%%%%%%%%%%%%%%%%%%%%%%%%%%%%%%%%%%%%%%%%%%%%%%%%%%%%%%%%%%%%%%%%%%%
\section{Proof of Main Results}
%%%%%%%%%%%%%%%%%%%%%%%%%%%%%%%%%%%%%%%%%%%%%%%%%%%%%%%%%%%%%%%%%%%%%%%
For $A\subset V$ define $\tau_{i}(A)$ as the hitting time of the $i$-th random walk $X_i$ to $A$, \ie
\begin{align*}
	\tau_{i}(A):=\min\{t\ge 0\mid X_{i}(t)\in A\}.
\end{align*}
If $A=\{x\}$, we simply write $\gt_i(x)$. Let $\cV_i(t)$ be the set of vertices of $V$ not visited by $X_i$ until time $t$. We call $\cV_i(t)$ the vacant set at time $t$. Note that the range $\cR_i(t)=V\setminus \cV_i(t)$ and the size of $\cV_i(t)$ can be written as
\begin{align*}
  |\cV_i(t)| = \sum_{x\in V}\vone_{\{\gt_i(x)>t\}},
\end{align*}
which will be frequently used in what follows.
We first give the proof of Theorem~\ref{thm: k-collab}.

\begin{proof}[Proof of Theorem~\ref{thm: k-collab}]
We present the proof in case of the discrete Markov chain, {\ie} cases~\ref{lazy version} and~\ref{even version}. The continuous case is analogous.

Let $y\in V$ be fixed and $|V|=n$. Let $\wh{P}$ be the matrix obtained from $P$ by removing the row and the column corresponding to $y$, that is, $\wh{P}(\xi, \eta) = P(\xi, \eta)$ for $\xi, \eta \in V \setminus\{y\}$. By reversibility, the matrix $A(x,\xi)=\pi(x)^{\half}\pi(\xi)^{-\half}\wh{P}(x,\xi)$ is symmetric.
It follows from the Spectral theorem that there exist eigenvalues $\gl_v$ and orthonormal eigenvectors $\varphi_v$ for $v\in V\setminus\{y\}$ for $A$.
%Furthermore, the eigenvalue interlacing theorem yields that if we enumerate the eigenvalues $\gl_1\le \gl_2\le \cdots \le \gl_{n}$ and $\wh{\gl}_1\le \wh{\gl}_2\le\cdots\le\wh{\gl}_{n-1}$ for $P$ and $\wh{P}$, then
%\begin{align*}
%  0\le \gl_1\le \wh{\gl}_1 \le \gl_2\le \wh{\gl}_2\le\cdots\le\wh{\gl}_{n-1}\le \gl_{n}\le 1.
%\end{align*}
By the spectral representation, we have
\begin{align*}
  \pr_x(\gt(y)>t) 
  &= \pr_x(X_s\neq y, s=0,1,\ldots,t)\\
  &= \sum_{\xi\in V\setminus\{y\}} \wh{P}^t(x,\xi)\\
  &= \sum_{\xi\in V\setminus\{y\}}\pi(\xi)^{\half}\pi(x)^{-\half}A^t(x,\xi)
  = \sum_{v, \xi\in V\setminus\{y\}}\pi(\xi)^{\half}\pi(x)^{-\half}(\gl_v)^t \varphi_v(x)\ol{\varphi_v(\xi)}.
\end{align*}
Thus, 
\begin{align*}
  \pr_{\pi}(\gt(y)>t) 
  &= \sum_{x,v, \xi\in V\setminus\{y\}} \gl_v^t\cdot \varphi_v(x)\ol{\varphi_v(\xi)} \pi(\xi)^{\half}\pi(x)^{\half}
  = \sum_{v \in V\setminus\{y\}}\ga_v \cdot \gl_v^t,
\end{align*}
where
\begin{align*}
  \ga_v = \Big|\sum_{\xi\in V\setminus\{y\}}\varphi_v(\xi)\pi(\xi)^{\half}\Big|^2\ge 0.
\end{align*}
Note that if $t=0$, we have $\pr_{\pi}(\gt(y)>t)=1-\pi(y)=\sum_v \ga_v$. Define a random variable $W$ by
\begin{align*}
  \pr(W=\lambda_v) = \frac{\ga_v}{1-\pi(y)},\quad\text{ for } v\in G\setminus\{y\},
\end{align*}
then
\begin{align*}
  (1-\pi(y))^{-1}\pr_{\pi}(\gt(y)>t) =\E[W^t].
\end{align*}
For the next step we need to show that $\E[W^t]\E[W^s]\,\le\, \E[W^{t+s}]$.

In cases~\ref{continuous version} and~\ref{lazy version} the corresponding eigenvalues are nonnegative. By the eigenvalue interlacing theorem, we see that $\gl_v$ are nonnegative for all $v\in G\setminus\{y\}$, that is, $W$ only takes nonnegative values. Let $t,s\in\dN$ and $W'$ be an independent copy of $W$, then it follows from the monotonicity of the maps $x^t$, $x^s$ for $x\ge 0$ that 
\begin{align*}
    0\le \E[(W^t-(W')^t)(W^s-(W')^s)] = 2(\E[W^{t+s}]-\E[W^t]\E[W^s]).
\end{align*}

Thus, we have that
%By H\"older inequality we get 
\begin{align}\label{ineq:moments}
  \frac{\pr_{\pi}(\gt(y)>t)\pr_{\pi}(\gt(y)>s) }{(1-\pi(y))^{2} }
  &= \E[W^t]\E[W^s]\,\le\, \E[W^{t+s}]
  = \frac{\pr_{\pi}(\gt(y)>t+s)}{1-\pi(y)}
\end{align}
%Since $\pr_{\wh{\pi}}(\gt(y)>t) =\frac{n^d}{n^d-1}\pr_{\pi}(\gt(y)>t)$,\textcolor{red}{(check this step for a general graph)} we have
%Therefore,
and so
\begin{align}\label{prob-ineq-1}
  \pr_{\pi}(\gt(y)>t) \pr_{\pi}(\gt(y)>s) 
  \le (1-\pi(y))\pr_{\pi}(\gt(y)>t+s)
  \le \pr_{\pi}(\gt(y)>t+s).
\end{align}
By iteration, we get
\begin{align}%\label{prob-ineq-2}
  \prod_{i=1}^k \pr_{\pi}(\gt(y)>t_i) 
  \le (1-\pi(y))^{k-1}\pr_{\pi}\left(\gt(y)>t_1+t_2+\cdots+t_k\right).
\end{align}

It remains to address the case~\ref{even version}. The odd values can be grouped in pairs if the sum $\sum_{i=1}^kt_i$ is even. For each pair of odd natural numbers $t$ and $s$ by monotonicity of the maps $x^t$, $x^s$ for $x\ge 0$, the inequality~\eqref{ineq:moments} holds. Hence replacing each pair of $\E[W^t]\E[W^s]$ by $\E[W^{t+s}]$ allows us to upper bound the product of such moments as another product of only even moments.
$$\prod_{i=1}^k \pr_{\pi}(\gt(y)>t_i) =(1-\pi(y))^{k-1}\prod_{i=1}^k \E[W^{t_i}]\le(1-\pi(y))^{k-1}\prod_{j=1}^{\ell} \E[W^{t_j}],$$
where all $t_j$'s are even. Thus we can replace $W$ by $\abs{W}$. The rest of the argument is the same as in cases~\ref{continuous version} and~\ref{lazy version}.
This completes the proof.
\end{proof}

\begin{proof}[Proof of Proposition~\ref{lem: nonstat}]
Since for each $i\in[k]$ we have that
\begin{align*}
  \left|\pr_{\nu_i}(\gt(y)>t)-\pr_{\pi}(\gt(y)>t)\right|
  &= \left|\sum_{x\in V}(\nu_i(x)-\pi(x))\pr_{x}(\gt(y)>t)\right|\\
  &\le \sup_{x\in V}\left|\frac{\nu_i(x)-\pi(x)}{\pi(x)}\right|\pr_{\pi}(\gt(y)>t).
\end{align*}
Assumption~\eqref{eq:nu-assumption1/3} implies that for $i\in[k],$ $$\pr_{\nu_i}(\gt(y)>t)\le (1+\pi_\ast)^{1/k}\cdot \pr_{\pi}(\gt(y)>t).$$ It then follows from~\eqref{prob-ineq-1} that for all $i$ and $j$
\begin{align}
\begin{split}\label{eq: gap}
  \prod_{i=1}^k\pr_{\nu_i}(\gt(y)>t_i) 
  &\le (1+ \pi_\ast)\cdot \prod_{i=1}^k\pr_{\pi}(\gt(y)>t_i)\\
  &\le \frac{1}{1-\pi(y)}\cdot \prod_{i=1}^k\pr_{\pi}(\gt(y)>t_i)
  \le \pr_{\pi}(\gt(y)>t_1+t_2+\cdots+t_k).
  \end{split}
\end{align}
By iteration, we complete the proof.
\end{proof}

\begin{proof}[Proof of Proposition~\ref{lem:almostind}]
We will prove the result for $k=2$; the general $k$ case follows similarly. For $x\in V$, Assumption~\eqref{eq:almost ind} and inequality~\eqref{eq: gap} yield that
\begin{align*}
  &\abs{\pr_{\mu}(\gt_1(x)>t_1, \gt_2(x)>t_2)-\pr_{\nu_1}(\gt_1(x)>t_1)\pr_{\nu_2}(\gt_2(x)>t_2)}\\
  &\qquad\le \sum_{y,z\in G}\pr_{y}(\gt_1(x)>t_1)\pr_{z}(\gt_2(x)>t_2) \cdot \abs{\mu(y,z)-\nu_1(y)\nu_2(z)}\\
  &\qquad\le (1-\eps)\cdot \pi_\ast\cdot \pr_{\pi}(\gt_1(x)>t_1)\pr_{\pi}(\gt_2(x)>t_2).
\end{align*}
By a similar argument to the one in the proof of Proposition~\ref{lem: nonstat} and Assumption~\eqref{eq:nu-assumption1/4}, we have that $$\pr_{\nu_i}(\gt(y)>t)\le \sqrt{1+ \eps\pi_\ast}\cdot \pr_{\pi}(\gt(y)>t).$$ 
Thus
\begin{align*}
  \pr_{\mu}(\gt_1(x)>t_1, \gt_2(x)>t_2)
  &\le \pr_{\nu_1}(\gt_1(x)>t_1)\pr_{\nu_2}(\gt_2(x)>t_2) \\
  &\qquad +\eps\pi_\ast\cdot \pr_{\pi}(\gt_1(x)>t_1)\pr_{\pi}(\gt_2(x)>t_2)\\
  &\le  (1+ \pi_\ast)\cdot  \pr_{\pi}(\gt_1(x)>t_1)\pr_{\pi}(\gt_2(x)>t_2)\le \pr_{\pi}(\gt(x)>t_1+t_2).
\end{align*}
This completes the proof.
\end{proof}

\begin{proof}[Proof of Lemma~\ref{thm: k-collab vs star}]
The case when $G$ is a general graph and $t=t_1=\cdots=t_k$ follows directly from Jansen's inequality. Indeed,
%\todo{Fix $x,y$. Let $f(t)=\pr_x(\tau(y))$. We need to bound $\prod_{i=1}^k f(t_i).$}

\begin{align*}
  \E_{x\sim\gn} \abs{\bigcup_{i=1}^k\cR_i(t)} 
  &= \abs{V}-\E_{x\sim \gn} \sum_{y\in V}\prod_{i=1}^k \vone_{\{\gt_i(y)>t\}} \\
  &= \abs{V}-\sum_{y\in V}\E_{x\sim \gn} (\pr_x(\gt(y)>t)^k\\
  &\le \abs{V}-\sum_{y\in V}(\pr_\gn(\gt(y)>t))^k
 = \E_{\gn^k} \abs{\bigcup_{i=1}^k\cR_i(t)} .
\end{align*}
%
\begin{comment}
First suppose $G$ is vertex-transitive. Fix a vertex $y\in V$. By the Markov property, one can enumerate the vertices of $G$ as
\begin{align*}
  V(y)=\{v_1, v_2, \cdots, v_n\}
\end{align*}
such that 
\todo{Check the conclusion below!!}
\[
a\le b\quad \Longleftrightarrow \quad\pr_{v_a}(\gt(y)>t)\le \pr_{v_b}(\gt(y)>t) \qquad\text{for all}\quad t\ge0.
\] 
This implies that $f_t(y,a):=\pr_{v_a}(\gt(y)>t)$ is non-decreasing in $a$, for each $t\ge 0$. The desired result then follows from the FKG inequality
\begin{align*}
  \E_{x\sim\gn} \abs{\bigcup_{i=1}^k\cR_i(t_i)} 
  &= n-\E_{x\sim \gn} \sum_{y\in V}\prod_{i=1}^k \vone_{\{\gt_i(y)>t_i\}} \\
  &= n-\sum_{y\in V}\E_{x\sim \gn} \prod_{i=1}^k \pr_x(\gt(y)>t_i) \le n-\sum_{y\in V}\prod_{i=1}^k \pr_\gn(\gt(y)>t_i)
 = \E_{\gn^k} \abs{\bigcup_{i=1}^k\cR_i(t_i)} .
\end{align*}
If $t=t_1=\cdots=t_k$, then the enumeration with monotonicity property is available for $t$, which leads to the desired result.
\end{comment}
\end{proof}

\begin{proof}[Proof of Theorem~\ref{thm:starone}]

Suppose $t_1\ge \ga T_{\mix}$ where $\ga>0$, that can depend on $n$, and $T_{\mix}$ is the mixing time. It is well-known that $T_{\mix}=\Theta(n^2)$. By the time reversal, it is equivalent to consider the ranges of two random walks $Y_1$ and $Y_2$ such that $Y_1$ starts from an ``almost'' stationary measure, $Y_1(0)=X_1(t_1)$, and $Y_2$ starts from $Y_2(0)=Y_1(t_1)$, see Figure~\ref{fig:timereversal}. If $t_1$ is large enough, it follows from the mixing property that the starting location $Y_2(0)$ is ``almost" stationary and is ``almost" independent from the random walk $Y_1(0)$. Indeed, since for some constant $c>0$ we have that $$\sup_{x\in V}|\pr_0(Y_1(t_1)=x)-\pi(x)|\le O(n^{-d}e^{-c\ga}),$$ it is easy to see that if $\ga=c'\log n$, for some $c'>0$, then the distributions of $Y_1(0)$ and $Y_2(0)$ satisfy the Assumptions~\eqref{eq:nu-assumption1/4} and~\eqref{eq:almost ind}. Thus, applying Proposition~\ref{lem:almostind} we conclude that three random walks starting at the same site on average cover the discrete torus more than a single random walk, as desired.\end{proof}
\tikzset{
		Big dot/.style={
			circle, inner sep=0pt, 
			minimum size=2.5mm, fill=black
		}
	}
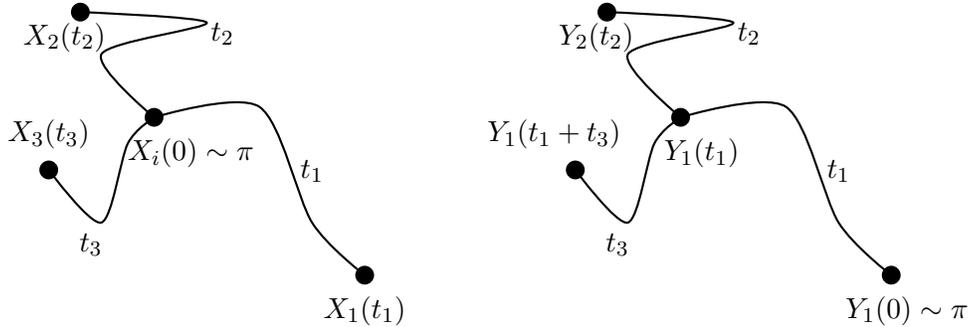
\begin{figure}[htb]
\begin{center}
	\begin{tikzpicture}[thick, scale=0.7]
	%Centers
	\node[Big dot] (1) at (-5,0) {};
	\node[Big dot] (2) at (5,0) {};
	   
	%End points of original	
	\node[Big dot] (3) at (-6.4,2) {};
	\node[Big dot] (V52) at (-1,-3) {};
	\node[Big dot] (V72) at (-7,-1) {};
	
	%lines of original
	\draw  plot [smooth]  coordinates {(-5,0) (-6,1.2) (-4,1.8) (-6.4,2)};
	\draw  plot [smooth]  coordinates {(-5,0) (-3,0.2) (-2,-2) (-1,-3)};
	\draw  plot [smooth]  coordinates {(-5,0) (-5.5,-0.5) (-6,-2) (-7,-1)};
	
	%names of original
	\draw (-4.3,-0.2) node[below] {$X_i(0)\sim \pi$};
	\draw (-1,-3.2) node[below] {$X_1(t_1)$};
	\draw (-6.7,2) node[below] {$X_2(t_2)$};
	\draw (-7,-0.8) node[above] {$X_3(t_3)$};
	
	%End points of reversal	
	\node[Big dot] (r3) at (3.6,2) {};
	\node[Big dot] (V52) at (9,-3) {};
	\node[Big dot] (V72) at (3,-1) {};
    %lines of reversal
	\draw  plot [smooth]  coordinates {(5,0) (4,1.2) (6,1.8) (3.6,2)};
	\draw  plot [smooth]  coordinates {(5,0) (7,0.2) (8,-2) (9,-3)};
	\draw  plot [smooth]  coordinates {(5,0) (4.5,-0.5) (4,-2) (3,-1)};
	
    %names of original
	\draw (5.4,-0.2) node[below] {$Y_1(t_1)$};
	\draw (9.3,-3.2) node[below] {$Y_1(0)\sim \pi$};
	\draw (3.4,2) node[below] {$Y_2(t_2)$};
	\draw (2.6,-0.8) node[above] {$Y_1(t_1+t_3)$};
	
    %timestamps
	\draw (-2,-0.6) node[below] {$t_1$};
	\draw (8,-0.6) node[below] {$t_1$};
    \draw (-6.2,-2)node[below] {$t_3$};
    \draw (3.8,-2)node[below] {$t_3$};
	\draw (-3.7,2)node[below] {$t_2$};
	\draw (6.3,2)node[below] {$t_2$};
    
    \end{tikzpicture}
\end{center}
\caption{Depiction of time reversal procedure from the proof of Theorem~\ref{thm:starone}. On the left there is a sketch of trajectories of $3$ independent random walks $X_i$ with lifespans $t_1,t_2$ and $t_3$ started at the same point sampled from stationary distribution. On the right there are two trajectories of random walks $Y_1$ and $Y_2$. $Y_1$ is started from a stationary point and has lifespan $t_1+t_3$, while $Y_2$ is started from $Y_1(t_1)$ and has a lifespan of $t_2$.}
\label{fig:timereversal}
\end{figure}

%%%%%%%%%%%%%%%%%%%%%%%%%%%%%%%%%%%%%%%%%%%%%%%%%%%%%%%%%%%%%%%%%%%%%%%
\section{Open questions and discussion}\label{sec: discussion}

In this section, we discuss further questions and state conjectures, some of which are based on the simulations done by Tyler M. Gall in the case when $G$ is a two-dimensional torus and Andrew Yin in the case $G$ is a random graph.

As mentioned in Section~\ref{sec:main} we believe that version~\ref{even version}, that states that the total lifespan $T:=\sum_{i=1}^kt_i$ of random walks $(X_i(t))_{t\in[t_i]}$ has to be even, is unnecessary and the result should hold for any $t_1, t_2,\ldots,t_k$. However, this assumption is used only in the inequality~\eqref{ineq:moments}. While it does not seem to be a drastic difference to allow the total time to be odd, the inequality becomes difficult to justify. Hence we leave this part as a conjecture.
\begin{conj}\label{conj: oddcase}
The inequality~\eqref{eq: main} holds for any total life time of the discrete Markov chains.
\end{conj}

It is natural to study the difference between the left-hand and right-hand sides of~\eqref{eq: main} as a function of the total life span $T$. When $T=0$, the difference is less or equal to $k-1$,
on the other hand, when $T$ is larger than the cover time of the graph $G$, the difference is $0$. 
\begin{figure}[htb]
  \centering
  \includegraphics[width=0.49\linewidth]{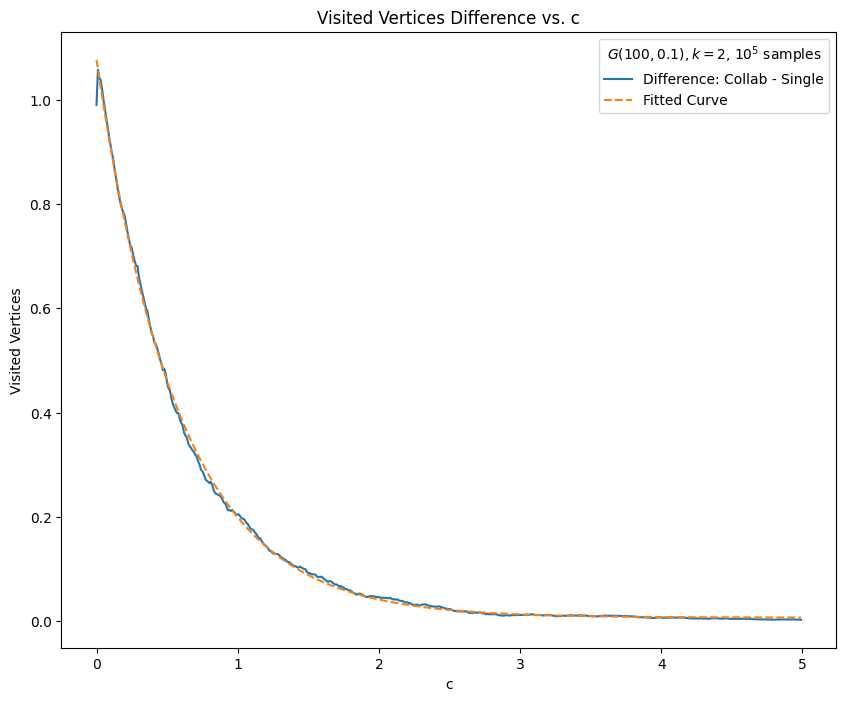}
  \includegraphics[width=0.5\linewidth]{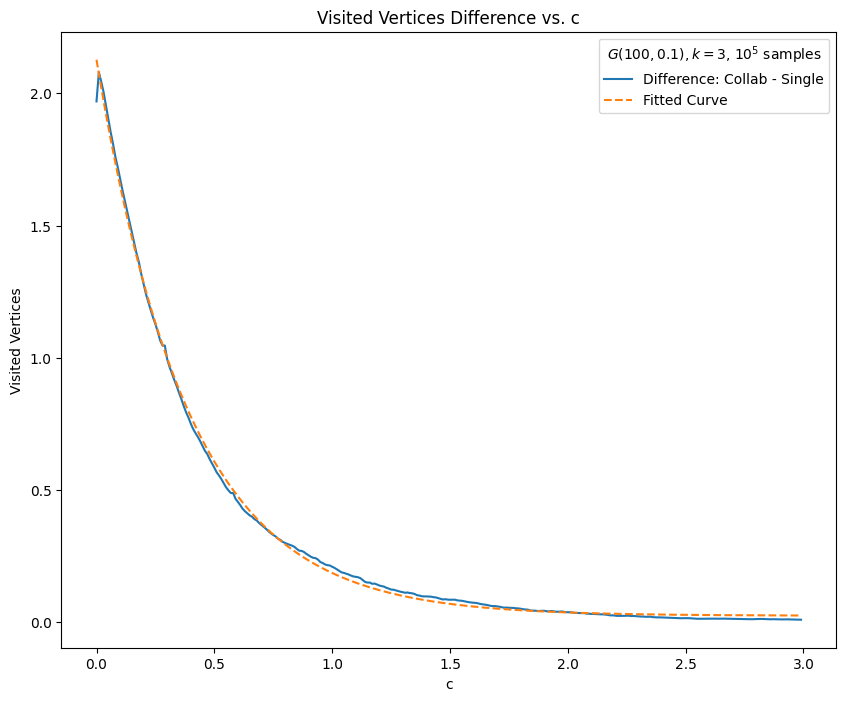}
  \caption{Simulation of the averaged difference of covered number of vertices (blue smooth curves) of $G_{n.p}$ with $n=100$, $p=0.1$ by $2$ random walks (left), $3$ random walks (right), of equal lengths and started at uniformly chosen vertex, and by a single random walk of length $T=k\cdot c\cdot n^2$ plotted versus $c$. Orange dashed curves are fitted exponential curves.}
  \label{fig:expdecay}
\end{figure}

Based on simulations, see Figure~\ref{fig:expdecay}, we ask the following question.
\begin{ques}
Under assumptions of Theorem~\ref{thm: k-collab} is it true that
\begin{align}\label{eq: exp}
  \E_{\pi^{k}} \abs{\bigcup_{i=1}^k\cR_i(t_i)} -\E_{\pi} \abs{\cR\left(\sum_{i=1}^kt_i\right)} \approx k\cdot \exp\left(-f(G)\cdot \sum_{i=1}^kt_i\right),
\end{align}
where $\pi$ is a stationary distribution, $f(G)$ is a function that depends on the graph $G$, and the number of walkers $k$.
\end{ques}

In addition to quantitative bound on the gap in the inequality~\eqref{eq: main} it is natural to study the fluctuation behaviors of the quantities on both sides of the inequality.
While the direct comparison between the variances of the two models is still open, we remark that in case of discrete torus $\dZ^d_n$, $d\ge 3$, partial answers are known.
The fluctuation behavior of the set of vertices covered by multiple random walks on the discrete torus $\dZ^d_n$, $d\ge 3$, as $n$ goes to $\infty$ was investigated in~\cite{DK21}. Indeed, if $t_1=t_2=\cdots=t_k=cn^d$, $c>0$, $d\ge 5$, and $\gs_{n,k}^2$ is the variance $\bigcup_{i=1}^k \cR_i(cn^d)$, then it was proven in~\cite{DK21} that $n^{-d}\gs_{n,k}^2$ converges to $\nu_d(2kc/G(0))$ where $\nu_d$ is an explicit function and $G(\cdot)$ is the Green's function on $\dZ^d$. Similar results hold for $d=3,4$. One can also extend it to general vertex-transitive graphs with some assumptions such as the hyper-cube $\dZ^n_2$ and the Cayley graphs of the symmetric groups, see~\cite{DK21}*{Remark 1.7}. The result in~\cite{DK21} implies that on the discrete torus $\dZ^d_n$, $d\ge 3$, the ranges of the two models in~\eqref{eq: main} has the same order of asymptotic variances if $t_1=\cdots=t_k=cn^d$ for some $c>0$.

We expect Lemma~\ref{thm: k-collab vs star} to hold more generally for arbitrary values of $t_i$ and a general class of graphs such as vertex transitive ones. For a transitive graph, it is clear that if for $x,y\in V$ we have that $\pr_x(\tau(y)\le t)$ is monotone with respect to $\dist_G(x,y)$ then the result follows from the FKG inequality. Such monotonicity does hold for simple graphs such as a cycle, however it is establish for a general graph. Hence we leave it as a question.

\begin{ques}\label{q:k-collab vs star}
Under what conditions on $k,t_1,t_2,\ldots,t_k\in \dN$, a graph $G$, and the probability measure $\nu$, we have that
\begin{align*}
  \E_{\gn^k} \abs{\bigcup_{i=1}^k\cR_i(t_i)} \ge\E_{\gn} \abs{\bigcup_{i=1}^k\cR_i(t_i)}?
\end{align*}
\end{ques}

Furthermore, it is of interest to understand how the distribution of the starting positions affects the performance compared to a single random walk with a combined lifespan. Theorem~\ref{thm: k-collab} and Proposition~\ref{lem: nonstat} state that if the starting distribution of each walker is close to stationary and is independent of each other, then such collaboration, on average, covers more of the graph than a single walk. On the other hand, if $k$ walkers start at a single vertex ({\ie} star shape), then for small values of $t_1,t_2,\ldots,t_k\in \dN$ they would interfere with each other, and hence a single walk with a lifespan $T=\sum_{i=1}^kt_i$ should visit more vertices. As values of $t_i$ increase and approach mixing time, the positions of the walkers become closer to independence and should start capturing graphs more and more efficiently. By that time, not much of the graph may be left unexplored. This suggests that one might want to consider the rate of acquiring new vertices and study the total time $T$, at which $k$ walkers start to explore faster than a single random walk.

\begin{ques}\label{q:star vs path}
Under what conditions on $k,t_1,t_2,\ldots,t_k\in \dN$ and for what measure $\mvgn$ on $V^k$, we have
\begin{align*}
  \E_{\mvgn} \abs{\bigcup_{i=1}^k\cR_i(t_i)} 
  \ge\E_{\pi} \abs{\cR\left(\sum_{i=1}^k t_i\right)}?
\end{align*}
\end{ques}

One can also look at the proposed problem from the point of view of the seeding problem, \ie~what initial conditions allow to cover the graph most efficiently. Sometimes in diffusion models on graphs, it is more beneficial to start at vertices with the highest degrees or at the furthest distance from each other in the graph metric. 
\begin{ques}\label{q:opt}
For a connected graph $G$, which measures $\nu_1\otimes\nu_2\otimes\cdots\otimes\nu_k$ starting at which $k$ random walks on average cover the graph most effectively?
\end{ques}

This article is concerned with inequalities between the average size of ranges of the corresponding strategy of capturing the graph. It is of great interest to extend this study to stochastic dominance under necessary assumptions if at all possible. 
\begin{ques}
    Let $t_1,t_2,\ldots, t_k\in\dN$ and $t=\sum_{i=1}^k t_i$. Is it true that for any $y>0$
    \begin{align*}\pr_{\pi^{k}}\left(\abs{\bigcup_{i=1}^k\cR_i(t_i)}\ge y\right) \ge\pr_{\pi}\left(|\cR(t)|\ge y\right)?
    \end{align*}
\end{ques}

Finally, it is likely that non-reversible Markov chains, such as non-backtracking random walk, would outperform random walks considered in this work. While, it seems intuitive that similar results ours should work, our analysis heavily relies on the reversibility. 
\begin{ques}
    Is it possible to extend Theorem~\ref{thm: k-collab} to non-reversible setting, \eg\ a non-backtracking random walk, various biased random walks, or self-avoiding random walks with re-sampling rule when the particle may not continue?
\end{ques}

\vskip.1in
\noindent{\bf Acknowledgments.} 
We thank our undergraduate students, Tyler M.~Gall and Andrew Yin, for helping with simulations and many insightful conversations. We also thank Illinois Geometry Lab at UIUC for providing a platform to undergraduate researchers.

\bibliography{ref-rw.bib}
\end{document}